\documentclass[12pt,reqno]{article}

\usepackage[usenames]{color}
\usepackage{amssymb}
\usepackage{graphicx}
\usepackage{amscd}
\usepackage{float}

\usepackage[colorlinks=true,
linkcolor=webgreen,
filecolor=webbrown,
citecolor=webgreen]{hyperref}

\definecolor{webgreen}{rgb}{0,.5,0}
\definecolor{webbrown}{rgb}{.6,0,0}

\usepackage{color}
\usepackage{fullpage}
\usepackage{float}

\usepackage{graphics,amsmath,amssymb}
\usepackage{amsthm}
\usepackage{amsfonts}
\usepackage{latexsym}
\usepackage{epsf}

\newcommand{\seqnum}[1]{\href{http://oeis.org/#1}{\underline{#1}}}

\begin{document}


\theoremstyle{plain}
\newtheorem{theorem}{Theorem}
\newtheorem{corollary}[theorem]{Corollary}
\newtheorem{lemma}[theorem]{Lemma}
\newtheorem{proposition}[theorem]{Proposition}

\theoremstyle{definition}
\newtheorem{definition}[theorem]{Definition}
\newtheorem{example}[theorem]{Example}
\newtheorem{conjecture}[theorem]{Conjecture}

\theoremstyle{remark}
\newtheorem{remark}[theorem]{Remark}

\begin{center}
\vskip 1cm{\LARGE\bf 
2-Free Tetranacci Sequences
}
\vskip 1cm
\large
Jeremy F.~Alm \\ 
Department of Mathematics\\
 Lamar University\\
 Beaumont, TX 77710\\
USA \\
\href{mailto:alm.academic@gmail.com}{\tt alm.academic@gmail.com} \\
\ \\

Taylor Herald\\
Department of Computer Science\\
 Washington University\\
 St. Louis, MO 63130\\
USA \\
\href{mailto: herald.r@wustl.edu}{\tt herald.r@wustl.edu} \\
\ \\

Ellen Rammelkamp Miller\\
Board of Trustees\\ 
Illinois College\\
Jacksonville, IL 62650 \\
USA \\
\href{mailto: ellenmiller7@gmail.com}{\tt ellenmiller7@gmail.com} \\
\ \\

Dave Sexton \\
Levi, Ray \& Shoup, Inc.\\
Springfield, IL 62704 \\
USA \\
\href{mailto: sexton.davey@mail.ic.edu}{\tt sexton.davey@mail.ic.edu}
\end{center}

\vskip .2in

\begin{abstract}
We consider a variant on the Tetranacci sequence, where one adds the previous four terms, then divides the sum by two until the result is odd.  We give an algorithm for constructing ``initially division-poor'' sequences, where over an initial segment one divides by two only once for each term.  We develop a probabilistic model that suggests that ``most'' sequences are unbounded, and provide computational data to support the underlying assumptions of the model.
\end{abstract}

\section{Introduction}
\label{intro}

In their 2014 paper \cite{AK}, Avila and Khovanova define $n$-free Fibonacci numbers via the recurrence,  \emph{add the  previous two terms and divide by  the largest power of $n$ dividing the sum}.  For example, a 2-free sequence beginning with $2,3$ is as follows:
\[
    2,3,5,1,3,1,1,\ldots
\]

A 3-free sequence beginning with $2,3$ is as follows:
\[
    2,3,5,8,13,7,20,9,29,\ldots
\]
Avila and Khovanova prove that 2-free sequences are eventually constant, and conjecture that 3-free sequences are all eventually periodic. 

Guy, Khovanova, and Salazar \cite{GKS} study a different variant of Fibonacci-like sequences that they call subprime Fibonacci sequences---a variant suggested by Conway.  To compute a term of a subprime Fibonacci sequence, one takes the sum of the previous two terms and, if the sum is composite, divides by its smallest prime divisor.  They study periodic subprime Fibonacci sequences and derive many interesting results; however, they are unable to prove that any such sequence has infinite range.  Indeed, it is difficult to imagine how one might prove such a thing.  The question, ``Do all subprime Fibonacci sequences eventually end in a cycle?"  may belong to the class of extremely difficult (possibly even formally unsolvable) problems  \cite{Conway}; one such example is the generalized Collatz problem, which was shown  to be undecidable \cite{KS}.

The first two authors of the present paper studied a more tractable version of subprime-like Fibonacci sequences that we called \emph{prime Fibonacci sequences} \cite{MR3473261}. In this version, one takes the sum of the previous two terms and returns the smallest odd prime divisor of that sum.  We showed that every such sequence terminates in a power of 2,  but they can be extended infinitely to the left, i.e., ``backwards''.

In the present paper, we consider 2-free Tetranacci sequences, that is, sequences where one adds the previous four terms, then divides the sum by two until the result is odd. (See Definition \ref{def:main}.)  

\begin{definition}\label{def:main}
Let $a_1, a_2, a_3, a_4$ be odd positive integers. Then $(a_i)_{i=1}^\infty$ is a 2-free Tetranacci sequence if for all $i\geq 1$,
$$a_{i+4}=\frac{a_i + a_{i+1} + a_{i+2} + a_{i+3}}{2^{d_i}}$$
where $d_i$ is the largest integer such that $2^{d_i}$ divides $a_i + a_{i+1} + a_{i+2} + a_{i+3}$.
\end{definition}

In the rest of this paper, we try to determine the long-term behavior of these sequences.

\section{Periodic sequences}\label{periodic}
In this section, we consider possible periods of 2-free Tetranacci sequences.
\begin{theorem}
If a 2-free Tetranacci sequence has period $p < 5$, then it has period 1.
\end{theorem}

\begin{proof}
The cases $p=2$ and $p=4$ are quite straightforward, so we leave those proofs to the reader.

Suppose $a,b,c,a,b,c,\ldots$ is a period-3 sequence.  Then we have
the following:
\begin{align}
  a+b+c+a &=b\cdot 2^i \label{eq:1}\\
  b+c+a+b &=c\cdot 2^j \label{eq:2}\\
  c+a+b+c &=a\cdot 2^k  \label{eq:3}
\end{align}
From these equations we get $a+b+c=a\cdot 2^{k-2}+b\cdot
2^{i-2}+c\cdot 2^{j-2}$.  Let $\alpha=k-2,\beta=i-2$, and
$\gamma=j-2$, so we have
\begin{equation}\label{abc}
  a+b+c=a\cdot 2^\alpha +b\cdot 2^\beta +c\cdot 2^\gamma,
\end{equation}
with $\alpha,\beta,\gamma\geq -1.$ We cannot have all of
$\alpha,\beta, \text{ and }\gamma$ greater than zero, since that
would make the right-hand-side of (\ref{abc}) greater than the
left-hand-side.  Also, if any of $\alpha,\beta,\gamma$ is equal to
$-1$, then two of them are, since $a,b$, and $c$ are odd yet the
right-hand-side is an integer.  Thus we have three cases:
\begin{itemize}
  \item $\alpha=\beta=\gamma=0$
  \item $\alpha>0,\beta\geq 0,\gamma=0$
  \item $\alpha>0,\beta=\gamma=-1$
\end{itemize}
The second case is easy to rule out.  Suppose $\alpha>0,\beta\geq
0$, and $\gamma=0$.  Then we have $a+b+c=a\cdot 2^\alpha +b\cdot
2^\beta+c$, so $a+b=a\cdot 2^\alpha+b\cdot 2^\beta$.  Now if
$\beta>0$ we have a contradiction, so suppose $\beta=0$, so we get
$a=a\cdot 2^\alpha$, a contradiction.

Let us consider the third case, where $\alpha>0$ and
$\beta=\gamma=-1$.  Then $a+b+c=a\cdot 2^\alpha+\frac{b+c}{2}$, and
\begin{equation}\label{2a}
  2a+b+c=a\cdot 2^{\alpha+1}.
\end{equation}

Now since $\alpha>0$ we have $2c+a+b\equiv_8 0$, and since
$\beta=\gamma=-1$ we have $2a+b+c\equiv_4 2b+a+c\equiv_4 2$.
Rearranging (\ref{2a}) we have $b+c=2a(2^\alpha -1)$.  Hence
$2a+b+c=2a+2a(2^\alpha -1)=a\cdot 2^{\alpha+1}\equiv_4 0$ since
$\alpha>0$.  But this is a contradiction.

Therefore it must be the case that $\alpha=\beta=\gamma=0$. Then (\ref{eq:1})-(\ref{eq:3}) express the fact that each of $a$, $b$, and $c$ is a weighted average of the other two. Therefore $a=b=c$.
\end{proof}

Note that $3, 3, 1, 1, 1\dots$ is a period-5 sequence.  In fact, it appears to be possibly the only nontrivial periodic sequence: every period-5 sequence we found was of the form $3a,3a,a,a,a\dots$, and we found no sequences with period larger than 5. 


\section{Initially division-poor sequences}\label{divpoor}
In this section, we show that there are sequences for which $d_i=1$ over arbitrarily long initial segments.

Let $\alpha$ be the positive real root of $2x^4-x^3-x^2-x-1=0$, the characteristic equation for (\ref{eq:divpoor}) below.  (Note that $\alpha$ is irrational.)  We can approximate $\alpha$ by
a rational $r=\frac{p}{q}$ such that $|\alpha-r |\leq\frac{1}{q^2}$.
Now, if we never divided by a higher power of 2 than $2^1$, the recursion for
our sequence would be
\begin{equation}\label{eq:divpoor}
a_{n+4}=\frac{a_{n+3}+a_{n+2}+a_{n+1}+a_n}{2},
\end{equation}
 and hence $a_n$ would
grow like $c\cdot \alpha^n$ for some constant $c>0$.  We will use
this idea to build a sequence that has $d_i=1$ for some initial
segment, by using backward recursion to make $a_n$ approximate
$c\cdot \alpha^n$ over that initial segment.  We will control the
backward propagation of error by increasing $q$, hence making $r$ a
better approximation to $\alpha$.

\begin{theorem}
  For any $N>0$ there is a 2-free Tetranacci sequence $(a_i)$ 
   such that for all $i\leq N$, $d_i=1$.
\end{theorem}

\begin{proof}
  Imagine that for some $x>0$, $a_4\approx x$,
  $a_3\approx\frac{x}{\alpha}$, $a_2\approx\frac{x}{\alpha^2}$, and
  $a_1\approx\frac{x}{\alpha^3}$.  We wish to define $a_0=2\cdot
  a_4-[a_1+a_2+a_3]$.  Suppose that $a_4$ is within $\epsilon$ of
  $x$, $a_3$ within $\epsilon$ of $\frac{x}{\alpha}$, etc.  We want to
  bound the error for $a_0$, i.e., $|a_0-\frac{x}{\alpha^4}|$.
\begin{align*}
\left|a_0-\frac{x}{\alpha^4}\right| &=\left|2a_4-(a_3+a_2+a_1)-\frac{x}{\alpha^4}\right|\\
&=\left|2a_4-(a_3+a_2+a_1)-\left[2x-\left(\frac{x}{\alpha}+\frac{x}{\alpha^2}+\frac{x}{\alpha^3}\right)\right]\right|\\
&\leq 2|a_4-x|+\left|a_3-\frac{x}{\alpha}\right|+\left|a_2-\frac{x}{\alpha^2}\right|+\left|a_1-\frac{x}{\alpha^3}\right|\\
&\leq 5\cdot\epsilon.
\end{align*}

    This shows that error propagation is at most
  5-fold.  Now we define $a_1,a_2,a_3,$ and $a_4$.

  Let $r=\frac{p}{q}$ be such that
  $|\alpha-r|\leq\frac{1}{q^2}$.  For $k=q^3$, let
\begin{align*}
  a_1 &=2k+1,\\
  a_2 &=r 2k+1,\\
  a_3 &=r^2 2k+1,\\
  a_4 &=r^3 2k+1.
\end{align*}
Then for $n\leq 0$, we may define $a_n=2\cdot
a_{n+4}-(a_{n+3}+a_{n+2}+a_{n+1})$, so that
$$a_{n+4}=\frac{a_{n+3}+a_{n+2}+a_{n+1}}{2}.$$  Thus we may ``go
backward" for as many terms as we like, so long as the $a_n$'s are
positive.  Since the $a_n$'s are approximately $c\cdot \alpha^n$, to
ensure that $a_n>0$, we need $5^{n+1}\cdot\epsilon\leq c\cdot
\alpha^{-(n+1)}$, where $c=a_1=2k+1$ and $\epsilon$ is the max error
among the terms $a_2,a_3,a_4$, i.e., $\epsilon$ is the largest of
\begin{align*}
  &|a_2-\alpha a_1|,\\
  &|a_3-\alpha^2 a_1|,\text{ and}\\
  &|a_4-\alpha^3 a_1|.
\end{align*}
It is easy to check that $|a_4-\alpha^3a_1|$ is the largest, and
that
\begin{align*}
  |a_4-\alpha^3a_1| &=|r^32k+1-\alpha^3(2k+1)|\\
  &\leq 2k|r^3-\alpha^3|+|\alpha^3-1|\\
  &\leq 2k|r -\alpha||r^2+r \alpha+\alpha^2|+|\alpha^3-1|\\
  &\leq 2q^3\cdot\frac{1}{q^2}\cdot (3\alpha^2+o(1))+1.5\\
  &\leq (6+o(1))q.
\end{align*}
Hence $\epsilon \leq(6+o(1))q$.

Now we need $5^{n+1}\cdot\epsilon\leq a_1\cdot \alpha^{-(n+1)}$;  since we have
$5^{n+1}\cdot\epsilon\leq 5^{n+1}(6+o(1))q$, then
\begin{align}
  5^{n+1}(6+o(1))q\leq(2q^3+1)\alpha^{-n+1}
  &\Longleftrightarrow (5\alpha)^{n+1}\leq\frac{2q^3+1}{(6+o(1))q}\\
  &\Longleftrightarrow n+1\leq\log
  \left[\frac{2q^3+1}{(6+o(1))q}\right]/\log(5\alpha),\label{RHS}
\end{align}
and the right-hand side of (\ref{RHS})  clearly grows without bound.
\end{proof}

Practical implementation of the above algorithm presents some problems, since it is difficult to get arbitrarily good rational approximations to $\alpha$; the best approximation (in terms of decimal digits) we can get with Python is 
\[
1.3490344565611562810403256662539206445217132568359375
\]
The longest initially division-poor sequence we could obtain from this approximation was 59 terms long. This was achieved using

\begin{align*}
r &= [1; 2, 1, 6, 2, 2, 3, 1, 1, 1, 1, 2, 1, 1, 67, 1, 1, 1, 5, 3, 3] \\ 
&= \dfrac{217560407}{161271201}
\end{align*}
as the approximation to $\alpha$.  This approximation of $\alpha$ is accurate to 16 decimal places (actual error about $2\times 10^{-17}$), but has denominator with only nine digits. Although we were able to obtain more convergents for  $\alpha$, they did not increase the length of the initially division-poor sequence so obtained. 

\begin{figure}[H]
\centering
\includegraphics[width=5in]{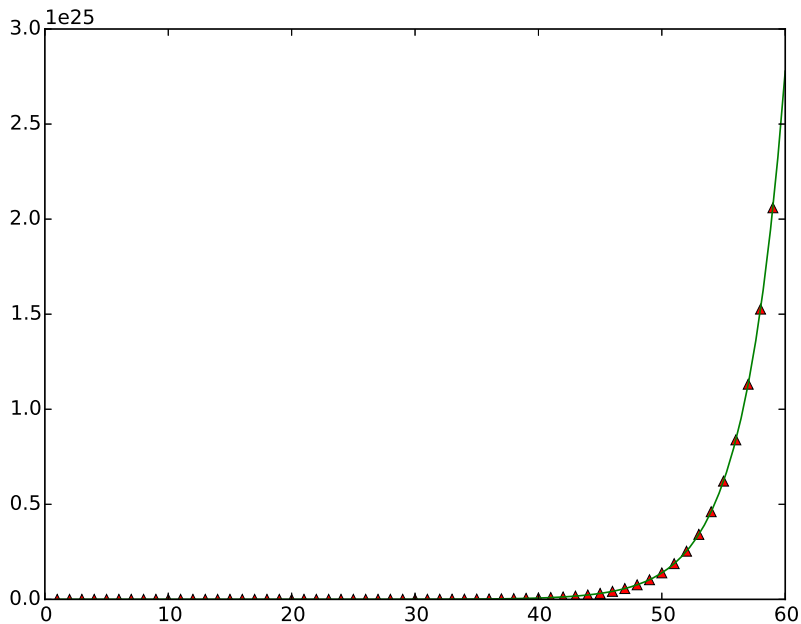}
\caption{59 division-poor terms, plotted against $y=438944974655058688r^t$}
\label{fig:divpoor}
\end{figure}

\section{Rates of growth and a probabilistic model}\label{growth}


\begin{theorem} 
Every 2-free Tetranacci sequence is either periodic or unbounded.
\end{theorem}

This result applies more generally to sequences determined completely by a few initial terms, and is doubtless not original, but we include a proof for completeness.

\begin{proof}
Suppose the range of $(a_n)$ is bounded.  We group the terms of $(a_n)$ into blocks of 4 terms, the first being $\{a_1,a_2,a_3,a_4\}$.  Since the range of $(a_k)$ is finite, there are but finitely many distinct blocks, so there must be a repetition, i.e., there are $i$ and $k$ so that the block $\{a_i,a_{i+1},a_{i+2},a_{i+3}\}$ is the same as the block $\{a_k,a_{k+1},a_{k+2},a_{k+3}\}$.  Since 4 consecutive terms determine the remainder of the sequence, $(a_n)$ is periodic with period dividing $k-i$.
\end{proof}

Now we would like to develop a probabilistic model to study the growth rate of sequences.  Our measure will be the ``average value" of 
\[ 
a_{k+4}-\frac{a_k,a_{k+1},a_{k+2},a_{k+3}}{4}. 
\]

\begin{theorem}
Let $(a_n)$ be periodic.  Then the average value of 
\[
a_{k+4}-\dfrac{a_k,a_{k+1},a_{k+2},a_{k+3}}{4}
\]
 is zero.
\end{theorem}
\begin{proof}
Let $(a_n)$ be periodic with period $p$.  Then the average value of 
\[
a_{k+4}-\dfrac{a_k,a_{k+1},a_{k+2},a_{k+3}}{4}
\]
is simply 
\[
\sum^{p-1}_{k=0}\left[a_{k+4}-\frac{a_k,a_{k+1},a_{k+2},a_{k+3}}{4}\right],
\]
where the indices are computed mod $p$.  It is straightforward to verify that the sum is zero.
\end{proof}

For an arbitrary sequence $(a_n)$, it is unclear how best to define the ``average'' value of 
\[ 
a_{k+4}-\frac{a_k,a_{k+1},a_{k+2},a_{k+3}}{4}. 
\]  
However, we can model $(a_n)$ by making the assumption that over all sequences $(a_n)$, ``half" the terms are equivalent to $1 \mod 4$, and ``half" are equivalent to $3 \mod 4$.  Similarly, we assume that each term has an equal probability of being equivalent to $1, 3, 5,$ or $7 \mod 8$, an equal probability of being equivalent to $1,3,5,7,9,11,13,$ or $15 \mod 16$, and so on.

A consequence of this is that given terms $a_1,a_2,a_3,a_4$, it follows that 
\[
a_5=\frac{a_1+a_2+a_3+a_4}{2^k},
\]
where $k$ is equal to $n$ with probability $\frac{1}{2^n}$ so half the time we divide by 2, a quarter of the time we divide by 4, an eighth of the time we divide by 8, and so on.  To see this, we will show that if $a_1+ a_2+a_3+a_4\equiv 0$ mod $2^k$, then the probability is $\frac{1}{2}$ that $a_1+ a_2+a_3+a_4\equiv 0$ mod $2^{k+1}$.  To keep the discussion concrete, we fix $k=3$, but the generalization is obvious.

So suppose $a_1+ a_2+a_3+a_4\equiv 0\pmod{8}$.  For each $i$, let $\alpha_i$ be $a_i$ reduced mod 8 (so $0\leq \alpha_i<8$).  Then $\alpha_1+\alpha_2+\alpha_3+\alpha_4\equiv 0\pmod{8}$.  Now for each $i$, either $a_i\equiv\alpha_i\pmod{16}$ or $a_i\equiv\alpha_i+8\pmod{16}$.  Let $\epsilon_i$ be $a_i-\alpha_i$ reduced mod 16, so $\epsilon_i=0$ or $\epsilon_i=8$.

Then if $\alpha_1+ \alpha_2+\alpha_3+\alpha_4\equiv 0\pmod{16}$, we have
\[
a_1+a_2+a_3+a_4\equiv 0\pmod{16}\text{ iff }\epsilon_1+\epsilon_2+\epsilon_3+\epsilon_4\equiv 0\pmod{16},
\]
and if $\alpha_1+\alpha_2+\alpha_3+\alpha_4\equiv 8\pmod{16}$,  we have
\[
a_1+a_2+a_3+a_4\equiv 0\pmod{16}\text{ iff }\epsilon_1+\epsilon_2+\epsilon_3+\epsilon_4\equiv 8\pmod{16}.
\]
Based on the probabilistic model, $\epsilon_1+\epsilon_2+\epsilon_3+\epsilon_4\equiv 0\pmod{16}$ with probability $\frac{1}{2}$, so half the time that $a_1+a_2+a_3+a_4\equiv 0\pmod{8}$, $a_1+a_2+a_3+a_4\equiv 0\pmod{16}$ as well.

Based on these probabilistic assumptions, we show that ``most" sequences are unbounded.

\begin{theorem}
Consider the probabilistic model, where 
\[
a_{k+4}=\frac{a_k+a_{k+1}+a_{k+2}+a_{k+3}}{2^{d_k}},
\]
and $d_k$ is equal to $j$ with probability $2^{-j}$.  Then the average value of 
\[
a_{k+4}-\frac{a_k+a_{k+1}+a_{k+2}+a_{k+3}}{4}
\]
is positive, so $(a_n)$ is unbounded.
\end{theorem}

\begin{proof}
Let $s=a_k+a_{k+1}+a_{k+2}+a_{k+4}$, so that 
\begin{align*}
a_{k+4}-\frac{a_k+a_{k+1}+a_{k+2}+a_{k+3}}{4} &=\frac{s}{2^j}-\frac{s}{4}\\
&=s\left(\frac{1}{2^j}-\frac{1}{4}\right).
\end{align*}
Then the average value is 
\[
s\sum^\infty_{i=1}\left(\frac{1}{2^i}-\frac{1}{4}\right)\cdot\frac{1}{2^i}=\frac{s}{12}.
\]
\end{proof}
Thus, on average the $k^{\text{th}}$ term is about 8.3\% larger than the average of the previous four terms.  This equates to a growth rate of about 3.3\% per term.

In order to check the reasonableness of our assumption, we calculated the first 1000 terms of each sequence beginning $a,b,c,d$ with $0<a,b,c,d<128$, and recorded the remainder mod 32.  See Figure \ref{fig:mod32}.

\begin{figure}[H]
\centering
\includegraphics[width=5in]{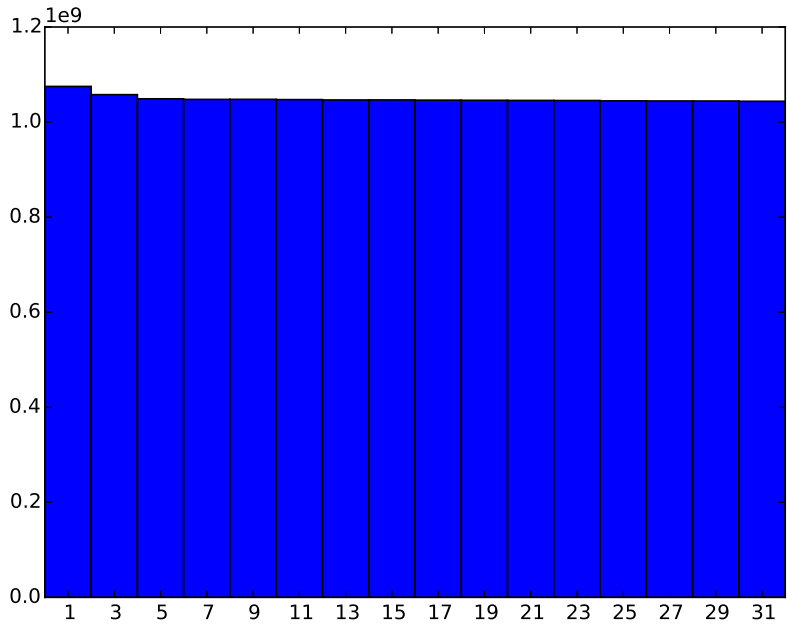}
\caption{Histogram of residues modulo 32 for $0<a,b,c,d<128$, 1000 terms each}
\label{fig:mod32}
\end{figure}

\section{Acknowledgments}
We thank Illinois College for their support of undergraduate research.  The second and fourth authors were undergraduates at IC when the bulk of this work was done.

\bibliographystyle{plain}
\bibliography{refs}










\bigskip
\hrule
\bigskip

\noindent 2010 {\it Mathematics Subject Classification}: Primary 11B39; Secondary 11B50.

\noindent \emph{Keywords: } Fibonacci number, Tetranacci number,
divisibility.

\bigskip
\hrule
\bigskip

\noindent (Concerned with sequences
\seqnum{A266295},
\seqnum{A000078},
\seqnum{A233526}, and
\seqnum{A233525}.)

\bigskip
\hrule
\bigskip

\end{document}